\documentclass{amsart}
\usepackage{amssymb, pstricks, wasysym}
\usepackage{hyperref}
\usepackage{upgreek}
\usepackage[T1]{fontenc}
\usepackage{textcomp}

\newtheorem{Theorem}{Theorem}[section]
\newtheorem{theorem}[Theorem]{Theorem}

\newtheorem{Lemma}[Theorem]{Lemma}
\newtheorem{lemma}[Theorem]{Lemma}
\newtheorem{Fact}[Theorem]{Fact}
\newtheorem{fact}[Theorem]{Fact}
\newtheorem{remark/def}[Theorem]{Remark/Definition}

\theoremstyle{definition}

\newtheorem{Remark}[Theorem]{Remark}
\newtheorem{remark}[Theorem]{Remark}
\newtheorem{Definition}[Theorem]{Definition}
\newtheorem{definition}[Theorem]{Definition}
\newtheorem{defn}[Theorem]{Definition}

\newtheorem{question}[Theorem]{Question}
\newtheorem{def/rem}[Theorem]{Definition/Remark}
\newtheorem{not/rem}[Theorem]{Notation/Remark}

\newcommand{\ind}[1][]{%
  \mathrel{
    \mathop{
      \vcenter{
        \hbox{\oalign{\noalign{\kern-.3ex}\hfil$\vert$\hfil\cr
              \noalign{\kern-.7ex}
              $\smile$\cr\noalign{\kern-.3ex}}}
      }
    }\displaylimits_{#1}
  }
}

\newcommand{\inds}[1][]{%
  \mathrel{
    \mathop{
      \vcenter{
        \hbox{\oalign{\noalign{\kern-.3ex}\hfil$\shortmid$\hfil\cr
              \noalign{\kern-.7ex}
              $\smile$\cr\noalign{\kern-.3ex}}}
      }
    }\displaylimits_{#1}
  }
}

\def \indk {\ind^K}

\def\lex{\operatorname{lex}}

\newcommand{\be}{\begin{enumerate}}
\newcommand{\bd}{\begin{defn}}

\newcommand{\bt}{\begin{theorem}}
\newcommand{\bl}{\begin{lemma}}
\newcommand{\ee}{\end{enumerate}}
\newcommand{\ed}{\end{defn}}
\newcommand{\et}{\end{theorem}}
\newcommand{\el}{\end{lemma}}

\newcommand{\la}{\langle}
\newcommand{\ra}{\rangle}

\newcommand{\CL}{{\mathcal L}}

\newcommand{\CM}{{\mathcal M}}

\newcommand{\CT}{{\mathcal T}}

\newcommand{\dom}{\mbox{dom}}

\def\dom{\operatorname{dom}}

\def\tp{\operatorname{tp}}

\title{Existence  in NSOP$_1$ theories}

\author{Byunghan Kim, Joonhee Kim, and Hyoyoon Lee}

\date{\today}

\address{Department of Mathematics\\
 Yonsei University\\
 Seoul, South Korea}
 \email{bkim@yonsei.ac.kr}
 
 \address{School of Mathematics\\
 Korea Institute for Advanced Study\\
 Seoul, South Korea}
 \email{kimjoonhee@kias.re.kr}
 
 \address{Center for Nano Materials\\
 G-LAMP\\
 Sogang University\\
 Seoul, South Korea \newline
 \indent Department of Mathematics\\
 Yonsei University\\
 Seoul, South Korea}
 \email{hyoyoonlee@sogang.ac.kr, hyoyoonlee@yonsei.ac.kr}

\begin{document}
\maketitle
\begin{abstract}
We show that Kim-forking satisfies existence  in all NSOP$_1$ theories. 
\end{abstract}

\bigskip

Recently NSOP$_1$ theories have been  significantly  studied 
 mainly due to that in NSOP$_1$ theories, nice properties such as symmetry, transitivity, and the Independence 
Theorem   hold over models with  Kim-independence, as  proved in \cite{KR20},\cite{KR21}. Later  in \cite{CKR20},\cite{DKR22},
the results are extended `over arbitrary sets' in any NSOP$_1$ 
theory with existence.
The class of NSOP$_1$ theories is considered as the proper superclass of that of simple theories, where  parallel  theoretic development  to that on simple theories is expected to be made. In particular all simple theories have existence. Therefore it is a burning question to know whether every NSOP$_1$ theory has
existence \footnote{Shortly after the first draft of this paper was presented, Mutchnik showed that there is an NSOP$_1$ theory that does not satisfy the existence axiom. See \cite{Mut24}.}. We investigate this problem in this paper and achieve a partial success. Namely we show that Kim-forking satisfies existence (over all sets) in all
NSOP$_1$ theories. More precisely, the failure of existence can not be witnessed by Morley sequences, in all NSOP$_1$ theories. 
This implies that any complete type in NSOP$_1$ $T$ has $\indk$-Morley sequences in it. Hence in such a theory, it makes sense to define 
a new kind of `dividing' of a formula witnessed by such a sequence (which is known to be equivalent to the usual `Kim-dividing' when the theory has existence),
 although we do not at all intend to push forward its study in this note. 
 
  In \S1, we collect basic definitions and facts necessary for the rest of the paper.  In \S2, preliminary lemmas  playing  key roles in 
  proving the main theorem are presented.  In \S3, we state and prove our main theorem that Kim-forking satisfies existence in all NSOP$_1$ $T$.
  We will discuss further related issues in \S4.

\section{Introduction} \label{intro}

In this section we summarize necessary definitions and facts (from  \cite{CR16}-\cite{Kim13}, etc.) which will be used freely throughout the paper. 
We work with a complete theory $T$ in a language $\CL$ with its  
large model $\CM$.
As usual,  (small) sets $A, B, C, \dots$ and (possibly infinite) tuples $a_i, b_i,\dots$  are from $\CM$ and models $M, N,\dots$ are small elementary submodels of $\CM$. 
 Given a sequence $(A_i\mid i<\kappa)$ and $j'<j<\kappa$, in addition to
$A_{<j}$, $A_{\geq j}$ and so on, we use $A_{j'<<j}$, $A_{\ne k}$ to denote  $(A_k\mid j'<k<j)$, $(A_i\mid i(\ne k)\in \kappa)$, respectively. 
\medskip 

We begin to recall basic definitions related to dividing and forking. We also recall
the notion of a {\em Morley sequence over $(A,B)$} originated from \cite{S}, which is useful in developing our arguments.  

 \begin{Definition}
 \be\item We say a formula $\varphi(x,a_0)$ {\em divides} over $A$ if there is an $A$-indiscernible sequence $\la a_i\mid i<\omega\ra$ such that $\{\varphi(x,a_i)\mid i<\omega\}$ is inconsistent.  A formula {\em forks} over $A$ if the formula implies some finite disjunction of formulas, each of which divides over $A$. Then we say a type divides/forks over $A$ if the type implies a formula which divides/forks over $A$.

 \item  We write $A\ind_BC$  if  $\tp(A/BC)$ does not fork over $B$.
 A sequence $I=\la a_i\mid i<\omega\ra$ is said to be $\ind$-{\em independent} (or just {\em independent}) over $A$ (or $A$-{\em independent}) if $a_i\ind_A a_{<i}$  for each $i<\omega$; if additionally $a_i\equiv_A a_0$ then we say it is $A$-independent 
 {\em in} $\tp(a_0/A)$. We say $I$ is {\em independent over $(A,B)$} if $a_i\ind_Aa_{<i}B$ for all $i\in\omega$.

\item  An $A$-independent  sequence $I=\la a_i\mid i<\omega\ra$ is said to be {\em Morley} over $A$ (or $A$-{\em Morley}) if  it is $A$-indiscernible.
 When  $I$ is Morley over $A$, we also call it  a Morley sequence {\em in} $\tp(a_0/A)$. $I$ is said to be {\em Morley over $(A,B)$}  if  it is independent over $(A,B)$ and $I$ is $AB$-indiscernible.

 \ee
 \end{Definition}

 We now  point out  the following handy fact which we will use  frequently.
 A proof can be found, say in \cite{Kim13}.

\begin{Fact} \label{e-r}
  Given $\kappa$, there is $\lambda= \lambda(\kappa)$ such that  for any sequence  $J=\la a_i: i<\lambda\ra$ of $\leq \kappa$-sized tuples
  and any set $A$ of size $\leq\kappa$, there is an $A$-indiscernible sequence 
  $\la b_i: i<\omega\ra$ which is {\em finitely based} on $J$ over $A$, i.e.    for any $b_{\leq n}$ there are
  $a_{i_{\leq n}}$ with $i_0<\dots <i_n\in \lambda$ such that $  b_{\leq n}\equiv_A a_{i_{\leq n}}$.
 \end{Fact}
 
Straightforward consequences from the definitions of dividing and forking 
are summarized as follows.  
 
 \begin{Fact}\label{basicsoffking}
 \be
 \item (Finite character) $A\ind_BC$ iff for any finite $a\in A$, $c\in C$ we have $a\ind_Bc$.
 
 \item (Base monotonicity)   If $D\ind_ABC$ holds then $D\ind_{AB}C$ follows. Hence if $I$ is Morley over $(A,B)$ then it is 
 Morley over $AB$ as well.
 
 \item (Left transitivity) If $C\ind_AB$ and $D\ind_{CA}B$  then we have $DC\ind_AB$. 
 
  \item (Extension)  
  For $A\subseteq B\subseteq C$, if $D\ind_AB$ then there is $D'\equiv_{B}D$ such that $D'\ind_AC$.

 \item Using extension and Fact \ref{e-r},
the equivalence of the following is obtained.
\be\item Any type over a set does not fork over the set.

\item For any $p\in S(A)$ there is a Morley sequence in $p$. 

\item Any consistent  formula $\varphi(x,a)$ does not fork over its finite parameter $a$.
\ee

\item (Existence) We say $T$ has {\em existence} if one of the  equivalent conditions in (5) holds.  Notice that $a\ind_MM$ holds for every model $M.$
Hence given any complete type over a model, there  is a Morley sequence in it. 

 \ee
  \end{Fact}

 As is well-known a typical example where existence (over $\emptyset$) fails to hold  is a circle with a ternary relation $R$ on the circle  such  that
 $R(a,b, c)$ holds iff 
 $a,c$ are not antipodal and $b$ is in the smaller arc connecting $a$ and $c$. We will discuss this example again in \S\ref{discussion}.

  \begin{Definition} \label{k-fkdef}
  \be
\item A formula $\varphi(x,a_0)$ {\em Kim-divides} over $A$ if there is an $A$-Morley sequence $\la a_i\mid i<\omega\ra$ such that $\{\varphi(x,a_i)\mid i<\omega\}$ is inconsistent.  A formula {\em Kim-forks} over $A$ if the formula implies some finite disjunction of formulas, each of which Kim-divides over $A$. A type Kim-forks over $A$ if the type implies a formula which Kim-forks over $A$. 
 
  \item We write $A\indk_BC$  if   $\tp(A/BC)$ does not Kim-fork over $B$. Obviously $A\ind_BC$ implies $A\indk_BC$. 
  
   A sequence $I=\la a_i\mid i<\omega\ra$ is said to be $\indk$-{\em independent}  over $A$  if $a_i\indk_A a_{<i}$  for each $i<\omega$; if additionally $a_i\equiv_A a_0$ then we say it is $\indk$-independent (over $A$) {\em in} $\tp(a_0/A)$. 
  
\item  Given an $A$-indiscernible sequence $\la a_i\mid i<\omega\ra$ we say it is  {\em $\indk$-Morley} over $A$ (or $A$-{\em $\indk$-Morley}) if it is $\indk$-independent over $A$; and we say it is a
 {\em total $\indk$-Morley} sequence over $A$ if $a_{\geq i}\indk_A a_{< i}$  holds for each $i<\omega$. Due to base monotonicity and left transitivity,
 a Morley sequence over $A$ is a total 
 Morley sequence over $A$, so a total
 $\indk$-Morley sequence over $A$ as well.
  
 \ee
 \end{Definition}

 In the definitions above if $A=\emptyset$  then we omit to state it. For example `$\varphi(x,a)$ (Kim-)divides/forks' means it does over $\emptyset$.
 
 \begin{Remark}  \label{algmorley}  
 \be
 \item Since we do not assume existence in Definition \ref{k-fkdef},  given formulas $\varphi(x,a)\models \psi(x,b)$, it is unclear whether  the following property holds:   
 if $\psi(x,b)$ Kim-divides over $A$ then so does $\varphi(x,a)$ over $A$. However this property holds with respect to Kim-forking (note that if $\varphi(x,a)\models \psi(x,b)$ and $\psi$ implies a finite disjunction of Kim-dividing formulas, then $\varphi$ also implies it). In particular, if $p(x)\models q(x)$ and $q(x)$ Kim-forks over $A$ then
 so does $p(x)$ over $A$; Kim-forking of a type over $A$
  only depends on its solution set; and  a type $p(x)$ Kim-forks over $A$ iff  $p(x)$ implies some finite disjunction of formulas, each of which Kim-divides over $A$.

\item (Finite character for $\indk$) $A\indk_BC$ iff for any finite $a\in A$, $c\in C$ we have $a\indk_Bc$.
  
  \item (Extension for $\indk$) 
  By the same compactness argument we have the following: 
   If a partial type $p$ over $C$ does not Kim-fork over a set then there is a completion of $p$ over $C$ which does not Kim-fork over the set. 
  In particular, for $A\subseteq B\subseteq C$, if $a\indk_AB$ then there is $a'\equiv_{B}a$ such that $a'\indk_AC$. 
  
  \item By the same techniques (using (3) and Fact \ref{e-r}) the equivalence of the following can be obtained.
  
  \be\item Any type over a set does not Kim-fork over the set.

\item For any $p\in S(A)$ there is an $\indk$-Morley sequence in $p$. 

\ee
We say Kim-forking  has {\em existence} if one of the above equivalent conditions holds.

 \ee
 \end{Remark}
 
 Now we recall  a theory being (N)SOP$_1$, or simple. 
 
\begin{definition} 
We say a formula $\varphi(x,y)\in\CL$ has {\em SOP$_{1}$} if there is a tree of tuples $\{a_{\eta}\mid \eta \in 2^{<\omega}\}$ such that 
\be
\item for all $\xi \in 2^{\omega}$, $\{\varphi(x,a_{\xi \lceil \alpha}) \mid  \alpha < \omega\}$ is consistent; and 
\item for all $\eta \in 2^{<\omega}$, if $\nu \unrhd \eta^\smallfrown \langle 0 \rangle$, then $\{\varphi(x;a_{\nu}), \varphi(x;a_{\eta^\smallfrown \langle 1\rangle })\}$ is inconsistent,
\ee
where $\unlhd$ denotes the tree partial order on $2^{<\omega}$.  We say $T$ has {\em SOP$_{1}$} if some formula has SOP$_{1}$ in $T$, and say  $T$ is {\em NSOP$_{1}$} if $T$  {\em does not have} SOP$_1$. 
\end{definition}

\begin{Fact}\cite{CR16},\cite{KR20} \label{crcrit} $T$ has  SOP$_1$ iff there is a formula $\varphi(x,y)$ and an indiscernible sequence $(b_ic_i\mid i<\omega)$ 
such that   for each $i<\omega$, $b_i\equiv_{(bc)_{<i} }c_i$ (*);
 and   $\{\varphi(x, b_i)\mid i\in \omega\}$ is consistent; while 
 $\{\varphi(x, c_i)\mid i\in \omega\}$ is inconsistent.
  (Notice that by compactness the same holds even if we replace (*) by $b_i\equiv_{(bc)_{>i} }c_i$.)
\end{Fact}

 We say $T$ is {\em simple} if $\ind$ satisfies local character: For any set $A$ and finite  $a$, there is $A_0\subseteq A$ with $|A_0|\leq |T|$ such that $a\ind_{A_0}A$. 
 As is well known,
  in simple $T$, existence holds; and  dividing, forking, and Kim-forking are all equal over any set. Any simple $T$ is NSOP$_1$. 
   It is not known yet whether existence holds in all NSOP$_1$ theories. In this note, however,  we show that Kim-forking satisfies existence in every NSOP$_1$ theory. 
 We will only  use the following two facts holding  in any NSOP$_1$ theory.

 \begin{Fact} \label{ess.facts} \cite{KR20}
 Assume $T$ is NSOP$_1$.  
  \be
\item  A formula Kim-divides over a model $M$ iff the formula Kim-forks over $M$. Moreover,   $c\indk_MM$  for any model $M$ (since $c\ind_MM$).
  
\item (Kim's lemma for Kim-dividing  over a model) A formula $\varphi(x,a_0)$ Kim-divides over a model $M$ iff for any Morley sequence $\la a_i\mid i<\omega\ra$ over $M$,  $\{\varphi(x,a_i)\mid i<\omega\}$ is inconsistent. 
 \ee
 \end{Fact}

We state the following folklore which comes mainly  from  extension.

\begin{fact}  \label{indextflore}  
Given a cardinal  $\kappa$, let $I=\la a_i\mid i<\kappa\ra$ be $\ind$-independent over $(A,B)$. 
\be\item For any $C$ there is $C'\equiv_{AB} C$ such that 
$I$ is independent over $(A,BC')$. It follows that $I\ind_ABC'$ and $I$ is $ABC'$-independent. 

\item Suppose additionally  $\kappa$ is infinite and  $I$ is Morley over $(A,B)$. Then again for  any $C$, there is $C'\equiv_{AB}C$ such that $I$ is Morley over $(A,BC')$.

\ee
\end{fact}
\begin{proof} (1) It suffices to verify the first statement, since it  implies  the second one due to base monotonicity and left transitivity. Given
$i<\kappa$, inductively assume that we have found $(b_j\mid j<i)$ such that $b_{<i}\equiv_{AB}a_{<i}$ and
$(b_j\mid j<i)$ is independent over $(A, BC)$ (*). It is enough (due to an $AB$-automorphism)  to find $b_i$ such that $b_{\leq i}\equiv_{AB}a_{\leq i}$ and
$(b_j\mid j\leq i)$ is independent over $(A, BC)$.  Now by (*),  there clearly  is  $b'_i$ such that $b_{<i}b'_i\equiv_{AB} a_{<i} a_i $,
so that $b'_i\ind_Ab_{<i}B$. Then by extension  there is  $b_i\equiv_{Ab_{<i}B} b'_i$ such that $b_i\ind_Ab_{<i}BC$.   
As wanted, $b_{\leq i}\equiv_{AB}b_{<i}b'_i\equiv_{AB}a_{\leq i}$ and $b_{\leq i}$ is independent over $(A,BC)$.

\medskip

(2) Due to compactness, we can stretch  out $I_\omega:=\la a_i\mid i<\omega\ra(\subseteq I)$ to  $I_\lambda=\la a_i\mid i<\lambda\ra$ for some sufficiently large $\lambda$ (to apply 
Fact \ref{e-r}), which is still Morley over $(A, B)$.  Then by (1), there  at least  is $C_0\equiv_{AB}C$ such that $I_\lambda$ is independent over $(A, BC_0)$. 
Now we apply Fact \ref{e-r} to obtain $ABC_0$-indiscernible  $J=(b_i\mid i<\omega)$ finitely based on $I_\lambda$ over $ABC_0$, so that clearly 
$J$ is Morley over $(A, BC_0)$.  Notice that $J\equiv_{AB} I_{\omega}$, since $I_\omega$ is $AB$-indiscernible. Hence there is  $C_1$
such that $C_1\equiv_{AB}C_0$ and $I_\omega$ is Morley over $(A, BC_1)$. Again by compactness we can extend $I_\omega$ to $I'$ of length $\kappa$ such that
$I'$ is Morley over $(A, BC_1)$. It also follows $I\equiv_{AB}I'$ and by an $AB$-automorphism we find a desired $C'\equiv_{AB}C_1$ such that
the original $I$ is Morley over    $(A, BC')$.
\end{proof}

Throughout this note we give the lexicographic order to $\mu\times\lambda$ so that for an array 
$(c_{ij}\mid (i,j)\in \mu\times \lambda)$, $c_{<ij}$ means $\{c_{i'j'}\mid i'<i \text{, or } (i'=i \text{ and } j'<j  )  \}$, and so on.

\begin{definition} Assume an array   $I=(c_{ij}\mid i,j\in \omega)$ is given with $I_i:=(c_{ij}\mid j\in \omega)$.
\be\item Recall that  $I$ is said to be {\em strongly  indiscernible} over $A$ 
if  $(I_i\mid i<\omega)$ is $A$-indiscernible and each $I_i$ is 
   $AI_{\ne i}$-indiscernible.

\item  $I$  is said to be a  {\em Morley array} over $A$ if it is strongly indiscernible over $A$, and 
 $c_{ij}\ind_Ac_{<ij}$ holds for each  $(i,j)\in \omega\times \omega$.
 \item $I$  is  a {\em weak* Morley array}  over $A$ if $c_{ij}\ind_Ac_{<ij}$ holds for each  $(i,j)\in \omega\times \omega$,
 each $I_i$ is $A$-indiscernible;  and if additionally   $(I_i\mid i\in \omega)$ is an $A$-Morley sequence then we call $I$ a {\em weak Morley array} over $A$.

 \ee
\end{definition}

If  $I$  is a Morley array over $A$, then mainly by left transitivity  and base monotonicity,
$I$ is  a weak Morley array.

\section{Preliminaries}

In this section we will present a couple of key lemmas critical in showing the main result Theorem  \ref{main}. 
Sophisticated applications of array and tree indiscernibilities keeping $\ind$-independence  
enable us to prove those lemmas.

\begin{Lemma} \label{gettingwmly}
Assume that $I=(c_{ij}\mid i,j\in \omega)$ is an  $A$-Morley array. Then for any $(A\subseteq)B$, there is $I'\equiv_AI$ such that 
$I'$ is a weak $B$-Morley array.
\end{Lemma}
\begin{proof}  
We  choose  $\mu,\kappa$ to  be the cardinals $\lambda(|B|+|I|)$  
and $\lambda(\mu)$, respectively,  as described in Fact \ref{e-r}. By compactness we can stretch out  $I$  to  $(c_{ij}\mid i\in \mu, \ j\in \kappa )$  
while keeping strong  indiscernibility over $A$. Then it must be $A$-Morley as well  (*).   
By induction on $\mu$ we will construct the following sequences.

\medskip

\noindent {\em Claim \ref{gettingwmly}.1.}  For each $i<\mu$ there is $J_i=(c'_{mj}\mid m\leq  i;\  j\in \omega)$ such that:
\be\item[(i)]    $J_i\equiv_A(c_{mj}\mid  m\leq i;\  j\in \omega)$ and for $i'<i$,  $J_{i'}\subseteq J_i$;

\item[(ii)]  $J_i$ is a weak* $B$-Morley array. 
\ee

\medskip

\noindent {\em Proof of Claim \ref{gettingwmly}.1.}  Given $i<\mu$, inductively assume that  for all $i'<i$, we have found $J_{i'}$ satisfying (i),(ii). 
Now put $J_{<i}:=\bigcup_{i'<i} J_{i'}=(c'_{mj}\mid m<i,\ j<\omega)$ (so if $i$ is a successor then $J_{<i}=J_{i-1}$). 
Then due to (i), there is $J'_i=(d'_{ij}\mid j<\kappa)$ such that $J_{<i}J'_i\equiv_A(c_{mj}\mid  m< i,\  j\in \omega)(c_{ij}\mid j\in\kappa )$.

In particular, due to (*) above,  $d'_{ij}\ind_Ad'_{i,<j}J_{<i}$ holds for each $j<\kappa$, that is 
$J'_i$ is independent over $(A, J_{<i})$.  Thus  by Fact \ref{indextflore}(1) with an $AJ_{<i}$-automorphism, 
there is $J^*_i=(d_{ij}\mid j<\kappa)$  such that  $J^*_i\equiv_{J_{<i}A} J'_i$ (so that 
 $J_{<i}J^*_i\equiv_A(c_{mj}\mid  m\in i,\  j\in \omega)(c_{ij}\mid j\in\kappa )$ as well (**)), while  $d_{ij}\ind_BJ_{<i}d_{i,<j}$  ($\dag_j$) holds for each $j<\kappa$. 

Finally, due to our choice of $\kappa$, we  can apply Fact \ref{e-r} to produce a desired $L=(c'_{ij}\mid j<\omega)$  which is finitely based on $J^*_i$ over $J_{<i}B$. 
Then due to (**) and strong indiscernibility over $A$,  clearly  $J_i:=J_{<i}L $ 
satisfies (i). Moreover $L$ is $B$-indiscernible, thus by  ($\dag_j$)   and  the induction hypothesis for $J_{<i}$, 
 (ii) for $i$ holds  too. 
\qed

\medskip

We put $J:=\bigcup_{i<\mu} J_i=((c'_{ij}\mid j\in \omega)\mid i\in \mu)$ as in Claim \ref{gettingwmly}.1. Owing to the choice of $\mu$, we can use  Fact \ref{e-r} to find 
 $B$-indiscernible $I' =((c''_{ij}\mid j\in \omega)\mid i\in \omega)$ finitely based on $J$ over 
$B$.  As wanted,    $I'\equiv_A I$ because of strong   indiscernibility. Now  by left transitivity and base monotonicity, $I'$ is a weak $B$-Morley array.
\end{proof}

We begin to recall the notions of  trees and tree indiscernibility  from \cite{KKS14},\cite{KR20},\cite{DKR22}.
 In this paper, by a {\em tree} we mean a set $\CT$ with a partial order $\unlhd$ such that for every $x\in \CT$, the set $\{y \mid y \unlhd x\}\subseteq \CT$ is linearly ordered by $\unlhd$.  For any $x,y\in\CT$, if there is a $\unlhd$-greatest element $z  \unlhd x,y$,  then  we call $z$ the {\em meet} of $x$ and $y$, and write it as $x\wedge y$.  Trees in this paper also possess suitable notions of {\em level} and $<_{\lex}$. For example  for  ordinals $\alpha, \gamma$,  we can see  $\gamma^{\leq \alpha}$ as a tree in the language $L_{s,\alpha}:=\{ \unlhd, \wedge,<_{\lex}, (P_{\beta})_{\beta \leq \alpha} \}$  with the usual interpretations of 
 $\unlhd, \wedge,<_{\lex}$, and  $P_{\beta}=\{\eta\in \gamma^{\leq \alpha}\mid \dom(\eta)=\beta \} $, the set of elements with level $\beta$.
 In the rest,   all the  $L_{s,\alpha}$-structures  we will consider are only trees.

\begin{defn}  Suppose that  an $L_{s,\alpha}$-structure $I$ is given. 
\be\item
We say $(a_{i} : i \in I)$ from $\CM$ is {\em s-indiscernible} over $A$ if whenever 
$(s_{0}, \ldots, s_{n})$, $(t_{0}, \ldots, t_{n})$ are tuples from $I$ with 
$$
\text{qftp}_{L_{s,\alpha}}(s_{0}, \ldots, s_{n}) = \text{qftp}_{L_{s,\alpha}}(t_{0}, \ldots, t_{n}),
$$
then we have
$a_{s_{0}}\ldots a_{s_{n}}\equiv_A a_{t_{0}}\ldots a_{t_{n}}.$
\item
We say $(b_i\mid i\in I)$ is {\em  locally based} on $(a_{i} : i \in I)$ over $A$ if 
given any finite set of formulas $\Delta$ from $\CL(A)$ and a finite tuple $(t_{0}, \ldots, t_{n})$ from  $I$, there is a tuple $(s_{0}, \ldots, s_{n})$ from $I$ such that 
\[
\text{qftp}_{L_{s,\alpha}} (t_{0}, \ldots, t_{n}) =\text{qftp}_{L_{s,\alpha}}(s_{0}, \ldots , s_{n})
\]
and 
\[\text{tp}_{\Delta}(b_{t_{0}}\ldots b_{t_{n}}) = \text{tp}_{\Delta}(a_{s_{0}} \ldots a_{s_{n}}).\]
\item We say $I$-indexed trees have the {\em  modeling property} if given  any tree $(a_{i} : i \in I)$ indexed by $I$ and any set $A$ there is s-indiscernible $(b_{i} : i \in I)$
over $A$ locally based on  $(a_{i} : i \in I)$ over $A$.
\ee
\end{defn}

\begin{fact}\cite[Theorem 4.3]{KKS14}\label{modeling} 
In any $T$, 
every $\omega^{<\omega}$-indexed tree has the modeling property, where  $\omega^{<\omega}$ is considered  as an $L_{s,\omega}$-structure.
\end{fact}

We now recall from \cite{KR20},\cite{DKR22}  an $L_{s,\alpha}$-tree  $\mathcal{T}_{\alpha}$, while {\em making one change}   as suggested in \cite{Kim14}. Namely  
 we consider all $\beta\leq \alpha$ in the following definition (1) (so that every $\CT_\alpha$ tree has a root), in order to 
make induction steps more palatable when we produce such  trees of parameters possessing desired properties.

\begin{Definition}\label{talphatree}
Suppose $\alpha$ is an ordinal.  We define $\mathcal{T}_{\alpha}$ to be the set of functions $\eta$ such that 
\be
\item $\text{dom}(\eta)$ is an end-segment of $\alpha$ of the form $[\beta,\alpha)$ for $\beta\leq\alpha$,

\item $\text{ran}(\eta) \subseteq \omega$, and 

\item (finite support)  the set $\{\gamma \in \text{dom}(\eta) : \eta(\gamma) \neq 0\}$ is finite.    
\ee
We interpret $\mathcal{T}_{\alpha}$ as an $L_{s,\alpha}$-structure as follows.
\begin{itemize}

\item We put $\eta \unlhd \xi$ if $\eta \subseteq \xi$.  Write $\eta \perp \xi$ if $\neg(\eta \unlhd \xi)$ and $\neg(\xi \unlhd \eta)$.  

\item $\eta \wedge \xi := \eta\lceil {[\beta, \alpha)} = \xi\lceil {[\beta, \alpha)}$ where
 $$\beta = \text{min}\{ \gamma \leq \alpha : \eta\lceil {[\gamma, \alpha)} =\xi\lceil {[\gamma, \alpha)}\}$$
(by finite support,  $\beta$ is not a limit ordinal). 

\item Put $\eta <_{\lex} \xi$ if  $\eta \vartriangleleft \xi$ or, $\eta \perp \xi$ with $\text{dom}(\eta \wedge \xi) = [\gamma +1,\alpha)$ and $\eta(\gamma) < \xi(\gamma)$.

\item For  $\beta\leq \alpha$ and $\eta:[\beta,\alpha)\to \omega$ in $\CT_\alpha$, 
we write $|\eta|:=\beta$, the {\em level} of $\eta$.
Then we put  $P_{\beta} := \{ \eta \in \mathcal{T}_{\alpha} : |\eta| = \beta\}$.  Note that $P_{0}$ is the set of \emph{top} level elements, and $P_{\alpha}=\{\emptyset\}$ is the set of the root.  
\end{itemize}
\end{Definition}

\begin{definition}
Suppose $\alpha,\beta$ are  ordinals.  
\begin{enumerate}
\item (Concatenation)  Suppose that $\eta :[\beta,\alpha)\to\omega$  in $\CT_\alpha$ is given where 
$\beta\leq \alpha$.  For  $i < \omega$, we let $\eta^\smallfrown \langle i \rangle$ denote the function $\eta \cup \{(\gamma,i)\}:[\gamma,\alpha)\to\omega$ in $\CT_\alpha$, if $\beta=\gamma+1$.
In particular,  $\la i\ra:=\emptyset^\smallfrown\la i\ra$ is in $\CT_{\alpha}$ for {\em any successor} $\alpha$.

On the other hand we  let $\langle i \rangle^\smallfrown \eta $ denote the function $\eta \cup \{(\alpha,i)\}:[\beta,\alpha+1)\to\omega$, which is 
(due to our definition,   always) in $\CT_{\alpha+1}$.

\item (Canonical inclusions) If $\beta\leq \alpha$, we define the map $\iota_{\beta\alpha } : \mathcal{T}_{\beta} \to \mathcal{T}_{\alpha}$ by $\iota_{\beta\alpha }(\eta) := \eta \cup \{(\gamma, 0) : \gamma \in [\beta,\alpha)\}$. Note that due to finite support, if $\alpha$ is  a limit ordinal and  $\xi(\ne \emptyset)\in \CT_\alpha$, then there is $\beta <\alpha$ and $\eta\in \CT_\beta$ such that $\iota_{\beta\alpha}(\eta)=\xi$, and hence   $\CT_\alpha =\bigcup_{\beta<\alpha}
\iota_{\beta\alpha}(\CT_\beta)\cup\{\emptyset\}$.

\item (The all $0$'s path) Given $\beta \leq \alpha$,  $\zeta_{\beta}\in\CT_\alpha$ denotes the function with $\text{dom}(\zeta_{\beta}) = [\beta, \alpha)$ and $\zeta_{\beta}(\gamma) = 0$ for all $\gamma \in [\beta,\alpha)$.  
\end{enumerate}
\end{definition}

We recall that  due to compactness applied to Fact \ref{modeling} we have the following modeling theorem for $\CT_\alpha$ trees.

\begin{Fact}\label{tamodeling} ($T$ is any theory.)
For any ordinal $\alpha$, all $\CT_\alpha$-indexed trees have the modeling
property. 
\end{Fact}

In the rest we assume trees of parameters are indexed by $L_{s,\alpha}$-trees having the modeling property (which we may simply say `modeling' or `s-modeling').

\begin{Definition}
A sequence  $(A_i\mid i<\kappa)$ of trees of  parameters   is called {\em mutually s-indiscernible} over $B$ if for each $i<\kappa$, $A_i$ is s-indiscernible over
$BA_{\ne i}$. 
\end{Definition}

We will freely use the following remark throughout the paper.  

\begin{Remark}  \label{rkforindnmlyseq} \be
\item Let  $C$ be a tree of parameters. If $C$ has the unique root $c_\emptyset$ then $C$ is s-indiscernible over $A$ if and only if it is $s$-indiscernible over $Ac_\emptyset$. 
Now for any set $B$, due to s-modeling  there is $C'$ which is s-indiscernible over $B$ locally based on $C$ over $B$.
If already $C$ is s-indiscernible over $B'(\subseteq B)$ then it  easily  follows that    $C\equiv_{B'} C'$. 

\item
If a sequence $(A_i\mid i<\omega)$ of trees is  $B$-indiscernible  and mutually s-indiscernible over $B$, then it
is s-indiscernible over $B$.

\item For sets $A,B$ and a tree of parameters $C=(c_\eta\mid \eta\in I)$, suppose that $ A\ind B C$
and $C$ is s-indiscernible over $B$.  If $C'=(c'_\eta\mid \eta\in I)$ is s-indiscernible over $AB$ and  locally based on $C$ over $AB$, then still $A\ind BC'$ holds:   If not, then  there is 
a forking  formula $\varphi(x,b\bar c'_{\bar \eta})$ with finite $b\in B$ and a finite tuple $\bar c'_{\bar \eta}\in C'$, realized by finite $a\in A$.
Owing to local basedness, there is $\bar c_{\bar \nu}\in C$ such that $\varphi(a,b\bar c_{\bar \nu})$ holds,
where  $\text{qftp}(\bar \eta)\equiv_{L_{s,\alpha}} \text{qftp}(\bar \nu)$. Since $C'$ is s-indiscernible over  $AB$, we have $\bar c'_{\bar \eta}\equiv_{AB}\bar c'_{\bar \nu}$, in particular 
 $\models \varphi(a, b\bar c'_{\bar \nu})$ as well, and $\varphi(x,b\bar c'_{\bar \nu})$ forks. Now due to (1), it follows $C'\equiv_BC$. Hence  $\varphi(x,b\bar c_{\bar \nu})$ also  is a  forking formula realized by $a$,  which contradicts  $A\ind BC$.    
\ee
\end{Remark} 

The following lemma plays a key role in proving the main theorem.  To show the lemma,
we borrow techniques from \cite{CKR20}  in order to  produce s-indiscernible trees preserving suggested   independence relations.

\begin{Lemma}\label{gettingmseq} Suppose that $C_0=(a_\eta\mid \eta\in\CT_\alpha)$ is s-indiscernible and   $I= \la b_i\mid i<\omega\ra$ is Morley, where $b_0=a_\emptyset\in C_0$.
Then for any $B$, there is a sequence $L=((a_{\eta,i}\mid \eta\in\CT_\alpha)\mid i<\omega)$ and $B_0\equiv B$ such that:  
\be\item $C_0=(a_{\eta,0}\mid \eta\in \CT_\alpha)$;

\item $\la a_{\emptyset,i}\mid i<\omega\ra \equiv I$ and
$\la a_{\emptyset,i}\mid 0< i<\omega\ra $  is Morley over $(\emptyset, C_0B_0)$; and

\item  $L$ is  $B_0$-indiscernible and mutually  s-indiscernible over $B_0$ (so s-indiscernible over $B_0$ as well). 
\ee
\end{Lemma}
\begin{proof} We first  choose $\kappa$ large enough with respect to $|B|+|C_0|$, and extend $I$ to  $I_\kappa=\la b_i\mid i<\kappa\ra$ which is still Morley.

\medskip 

\noindent {\em Claim \ref{gettingmseq}.1.}  There are  $B_0\equiv B$ and $J=(d_i\mid i<\kappa)\equiv_{b_0} I_\kappa$  with $b_0=d_0$ such that 
\be
\item[(i)] $J$ is Morley over $(\emptyset, B_0)$;

\item[(ii)] $C_0$ is s-indiscernible over $B_0J$; and

\item[(iii)] $J\smallsetminus \{b_0\}$ is independent over $(\emptyset, C_0B_0)$.

\ee

\medskip

\noindent {\em Proof of Claim \ref{gettingmseq}.1.} By Fact \ref{indextflore}(2), there is $B'_0\equiv B$ such that $I_\kappa$ is Morley over 
$(\emptyset, B'_0)$. Now  by modeling, there is  $(a_\emptyset\in) C'_0$ such that $C'_0$  is s-indiscernible over $a_\emptyset B'_0$
locally based on $C_0$ over $a_\emptyset B'_0$.

Notice now  that $C_0\smallsetminus\{a_\emptyset\}$ is s-indiscernible over $a_\emptyset$. Thus by Remark \ref{rkforindnmlyseq}(1) we have  $C_0\equiv_{a_\emptyset} C'_0$. Hence there is $J'B_0$ such that   $J'B_0C_0\equiv_{a_\emptyset}I_\kappa B'_0C'_0$. In particular  $C_0$ is s-indiscernible over $B_0$ and  $J'$ is Morley   over $(\emptyset, B_0)$ (*). In order to further satisfy (ii) and (iii), we will find a suitable $J\equiv_{B_0} J'$ as follows.

Write $J':=(d'_i\mid i<\kappa)$ where $d'_0=b_0$.  Using induction on $\kappa$,  we will construct  $J=(d_i\mid i<\kappa)$ ($d_0=d'_0$) such that for each $0<i<\kappa$, 
\be\item[(a)]
$d_{\leq i}\equiv_{d_0B_0} d'_{\leq i}$; 

\item[(b)] $C_0$ is s-indiscernible over $d_{\leq i}B_0$; and

\item[(c)]
$d_i\ind B_0C_0d_{<i}$.
\ee 

Hence  given  $0<i<\kappa$, assume that we have obtained $d_j$ for   
all $(0<)j<i$ (as said $d_0=d'_0=b_0$) satisfying (a),(b),(c). Now by  extension with (a) and (*) above, there is $d^*_i$ such that  $d_{<i}d^*_i\equiv_{d_0B_0} d'_{\leq i}$ and $d^*_i\ind  d_{<i}B_0C_0$ (**). Then by modeling there is $C^*_0$ s-indiscernible
over $d_{<i}d^*_iB_0$ locally based on $C_0$ over $d_{<i}d^*_iB_0$ (***). By the induction hypothesis for (b) (even for $j=0$) and Remark \ref{rkforindnmlyseq}(1), we then have that $ C^*_0\equiv_{d_{<i}B_0} C_0$. Thus there is a wanted 
$d_i$ such that $d_{i} C_0\equiv_{d_{<i}B_0}d^*_iC^*_0$.  Therefore (a),(b) for $i$ hold.  Lastly,
due to Remark \ref{rkforindnmlyseq}(3) applied to (**) and (***) with the induction hypothesis for (b), we have $d^*_i\ind  d_{<i}B_0C^*_0$. Hence (c) for $i$ also holds.

In conclusion we have produced $J$; and due to (a),(b),(c), together with $B_0C_0$, Claim \ref{gettingmseq}.1 is verified.
\qed

\medskip

We continue to  work with $J=(d_i\mid i<\kappa)$ and $B_0$
 produced in Claim \ref{gettingmseq}.1.  We further claim the following.

\medskip

\noindent {\em Claim \ref{gettingmseq}.2.}  For each $j<\kappa$, there is 
$C_j$ such that 
\be\item[(i)$^*$]  $d_{\geq j}C_j\equiv_{B_0}JC_0$,

\item[(ii)$^*$] $\la C_0,\dots ,C_j\ra$  is mutually s-indiscernible over $d_{>j}B_0$. 

\ee

\medskip 

\noindent {\em Proof of Claim \ref{gettingmseq}.2.} We suppose that for every $j'<j$ we have obtained $C_{j'}$ with the properties holding. If $j=0$ then clearly  $C_0$ (with Claim \ref{gettingmseq}.1 (ii)) serves it. 

 Now consider the case $0<j$.  
 Due to (ii)$^*$ for all $j'<j$ we have that 
  $(C_{j'}\mid j'<j)$  is mutually s-indiscernible over $d_{\geq j}B_0$ ($\star$). 
  
 Moreover  there is 
 $D^*_j $ such that $d_{\geq j}D^*_j\equiv_{B_0} JC_0$, since  $J$ is $B_0$-indiscernible (Claim \ref{gettingmseq}.1(i)).
 Hence $d_j$ is the root of the tree $D^*_j$, and $D^*_j$ is s-indiscernible over $d_{\geq  j}B_0$, by Claim \ref{gettingmseq}.1(ii).  Then   by modeling we take  
 $D'_j$ which is, over  $d_{\geq  j}C_{< j}B_0 $,   s-indiscernible and locally based on $D^*_j$  ($\dag$).
 By Remark \ref{rkforindnmlyseq}(1)    we have  $D'_j\equiv_{d_{\geq j}B_0} D^*_j$ as well  ($\ddag$). 
 
 By now  we choose   $ (D_{j'}\mid j'<j)$ as follows.  
 If for all $k< j'(<j)$ we have selected  $D_k$, then  by modeling we  choose $D_{j'}$  
 s-indiscernible over $D_{<j'}  C_{j'< <j}   D'_{j}  d_{\geq  j}B_0$ locally based on $C_{j'}$.  We will show the following two properties.

\begin{enumerate}
    \item[(a)$^*$] The resulting $\{ D_k \mid k< j\}\cup \{ D'_j\} $ is mutually s-indiscernible over $d_{> j}B_0$.
    \item[(b)$^*$] $  D_{< j}\equiv_{ d_{\geq j}B_0}   C_{< j}.$   
\end{enumerate}

For  (a)$^*$, we firstly show that given  $k< j$, $D_k$ is s-indiscernible over $D_{\ne k}D'_jd_{> j}B_0$. To lead a contradiction let us deny this. 
Hence  there are   
$$k_0<\dots <k_i=k<\dots <k_n< j \ (i\leq  n);$$ 
$\varphi(\bar x_0, \dots, \bar x_n, \bar y)\in \CL(d_{> j}B_0)$; finite tuples of indices $\bar \eta_0,\dots, \bar \eta_n, \bar \mu_i, \bar \nu \in \CT_\alpha$ 
with 
$$ \text{qftp}(\bar \eta_i)\equiv_{L_{s,\alpha}} \text{qftp}(\bar \mu_i)  $$
and corresponding   tuples $\bar u^{0}_{\bar \eta_{0}}\in D_{k_0},\dots, \bar u^{n}_{\bar \eta_{n}}\in D_{k_n}$,  $\bar u^i_{\bar \mu_i }\in D_{k_i}$, $\bar w_{\bar \nu}\in D'_{j}$ such that 
 $$\models \varphi(\bar u^0_{\bar \eta_0}, \dots, \bar u^i_{\bar  \eta_i},\dots, \bar u^n_{\bar \eta_n}, \bar w_{\bar \nu})\wedge \neg \varphi(\bar u^0_{\bar \eta_0}, \dots;\bar u^i_{\bar  \mu_i};\dots, \bar u^n_{\bar \eta_n}, \bar w_{\bar \nu}).$$
By the construction of $D_{k_{i+1}},\dots, D_{k_n}$ and local basedness, we can consecutively  find 
$\bar v^n_{\bar \nu_n}\in C_{k_n}, \dots, \bar v^{i+1}_{\bar \nu_{i+1}}\in C_{k_{i+1}}$ (where of course $ \text{qftp}(\bar \eta_n)\equiv_{L_{s,\alpha}} \text{qftp}(\bar v_n)  $, and so on)
such that 
$$\varphi(\bar u^0_{\bar \eta_0}, \dots, \bar u^i_{\bar  \eta_i}; \bar v^{i+1}_{\bar \nu_{i+1}},\dots, \bar v^n_{\bar \nu_n}, \bar w_{\bar \nu})\wedge \neg \varphi(\bar u^0_{\bar \eta_0}, \dots; \bar u^i_{\bar  \mu_i};\bar v^{i+1}_{\bar \nu_{i+1}},\dots, \bar v^n_{\bar \nu_n}, \bar w_{\bar \nu})$$
holds. But then due to the s-indiscernibility  of $D_k=D_{k_i}$ over the rest of the parameters, this leads to a contradiction.

Secondly, by the similar arguments with ($\dag$) above, it also easily follows that $D'_j$ is s-indiscernible over $D_{< j}d_{> j}B_0$. We have shown (a)$^*$.

Now (b)$^*$ can be shown  straightforwardly,  by iterated applications of Remark \ref{rkforindnmlyseq}(1) using   ($\star$) above and the way of choosing $D_{<j}$.

We  can now return our attention to finding wanted $C_j$.   Due to (b)$^*$ above, there is  $C_j$  such that 
$$   C_{< j}  C_j \equiv_{ d_{\geq j} B_0} D_{< j}D'_j.$$ 
Then  due to above ($\ddag$) and (a)$^*$,
it follows that $C_j$ and $C_{\leq j}$ satisfy  (i)$^*$,(ii)$^*$ for $j$. We have shown Claim \ref{gettingmseq}.2.
\qed
\medskip

We are about  to finish our proof of   Lemma \ref{gettingmseq}. We work again with the selected $J=(d_i\mid i<\kappa)$ and $B_0$ from Claim \ref{gettingmseq}.1, 
and $\la C_i\mid i<\kappa\ra $ from Claim \ref{gettingmseq}.2.  
In  particular,  $d_{\geq i}C_i\equiv_{B_0}JC_0$ and  $(C_i\mid i<\kappa)$ is
mutually s-indiscernible over $B_0$ ($\sharp$).
Now since $\kappa$ is sufficiently large, by Fact \ref{e-r}, there is   $B_0$-indiscernible 
$(C'_i\mid i<\omega)$
finitely based on $(C_i\mid i<\kappa)$ 
over $B_0$. By an automorphism we can assume $C'_0=C_0$.
Then we put a wanted $L$ as $(C'_i\mid i<\omega)$.   Clearly (1) holds. Since  $J\equiv I_\kappa$  (Claim \ref{gettingmseq}.1), 
the first condition of (2) is satisfied. The second condition of (2) holds as well, due to that $(C'_i\mid 0<i<\omega)$ is $B_0C_0$-indiscernible, and 
Claim \ref{gettingmseq}.1(iii) with  ($\sharp$). Moreover (3) comes from ($\sharp$). The proof of Lemma \ref{gettingmseq} is completed. 
\end{proof}

\section{Existence and the main theorem}\label{mainthm}

In this section we consider the case  where Kim-forking fails to have existence. Hence after naming  parameters, we have consistent  formulas $\psi(x)\in \CL$ and 
 $\varphi_i(x,b_i)$   ($i\leq n$)  such that
$\models \psi(x)\rightarrow \varphi_0(x,b_0)\vee\dots \vee \varphi_n(x,b_n)$, and each   $\varphi_i(x,b_i)$  Kim-divides. Then  replacing 
$\varphi_i(x,b_i)$ by $\psi(x)\wedge \varphi_i(x,b_i)$ (and renaming it as $\varphi_i(x,b_i)$) allow  us   the following. 
\medskip 

{\bf In this section  we assume:} $$\models \psi(x)\leftrightarrow \varphi_0(x,b_0)\vee\dots \vee \varphi_n(x,b_n)$$ holds (all are consistent formulas) and 
each  $\varphi_i(x,b_i)$  ($i\leq n$) Kim-divides witnessed by a Morley sequence $I_i=(b_i^j\mid j<\omega)$ where $b_i^0=b_i$. We now put $D:=b_0\dots b_n$.
Notice that  for any $b'_i\equiv b_i$, $\varphi_i(x,b'_i)\models \psi(x)$, and for any $b'_0\dots b'_n\equiv D$, $ \psi(x)$ is equivalent to $\varphi_0(x,b'_0)\vee\dots \vee \varphi_n(x,b'_n)$ as well.

\begin{Lemma} \label{treeinducedffkfe} Given a cardinal $\kappa$ and   each ordinal $\alpha<\kappa$, there is  s-indiscernible $C_\alpha=(a_\eta\mid \eta\in \CT_\alpha)$ and $D_\alpha=b_0^\alpha\dots b_n^\alpha\equiv D$, and there is   $f: \kappa \to n+1$, which  satisfy the following.
\be
\item  Suppose that   $\alpha=\gamma+1$.   Then  
\be\item
$(a_{\la i\ra}\mid i\in \omega)\equiv I_{f(\gamma)}$ (where $\la i\ra \in \CT_\alpha$), and

\item for  each $0<i<\omega$, we have   that
$$a_{\la i\ra}\ind a_{\unrhd \la0\ra}(a_{ \la j\ra}\mid j<i)D_\alpha.$$
\ee

\item  $C_\alpha$ is  s-indiscernible over $D_\alpha$, and $a_\emptyset=b^\alpha_{f(\alpha)}\in D_\alpha$.

\item 
For every $\eta\in P_0$, $\{\varphi_{f(|\nu|)}(x, a_{\nu})\mid  \nu\unlhd \eta\}$ is consistent. 

\item  For $\alpha'<\alpha$, $C_{\alpha'}\subseteq C_\alpha$ via the  canonical inclusion. 
\ee
\end{Lemma}
\begin{proof}
Suppose that  we have determined  $C_\beta, D_\beta$ and $f(\beta) $ for all $\beta<\alpha$ satisfying the properties. We want to choose $C_\alpha,D_\alpha$ and $f(\alpha)$.

If $\alpha=0$ then choose $C_0$ to be any $b_i\in D_0:=D$  and $f(0)=i$.  This choice clearly works. 

Now if $\alpha$ is a limit ordinal, notice that $C_{<\alpha}$, the union 
 of $C_\beta$ for all $\beta<\alpha$ is also s-indiscernible, due to the canonical inclusion.
Then by modeling we  choose $C'_{<\alpha}$ which is s-indiscernible over $D$ and locally based on $C_{<\alpha}$ over $D$. Then since 
$C_{<\alpha}$ is s-indiscernible, we have $C_{<\alpha} \equiv C'_{<\alpha}$. Hence there is a wanted $D_\alpha=b_0^\alpha\dots b_n^\alpha$ such that  
$C_{<\alpha} D_\alpha\equiv C'_{<\alpha}D$, so $C_{<\alpha}$ is s-indiscernible over  $D_\alpha$ (*). Notice that by the induction hypothesis with finite support of $\CT_{<\alpha}$, given  $\eta\in P_0$ (of $\CT_{<\alpha}$),  $\{\varphi_{f(|\nu|)}(x, a_{\nu})\mid  \nu\unlhd \eta\}$ is also consistent. 
Moreover due to s-indiscernibility, for   $\nu\in P_\beta$, we have  $a_\nu\equiv b_{f(\beta)}$. 
Thus obviously 
$$ \{\varphi_{f(|\nu|)}(x, a_{\nu})\mid  \nu\unlhd \eta\}\models \psi(x).$$
Since  $\models \psi(x)\leftrightarrow \varphi_0(x,b^\alpha_0)\vee\dots \vee \varphi_n(x,b^\alpha_n)$ as well,    there is some $i_0\leq n$ such that 
$$ \{\varphi_{f(|\nu|)}(x, a_{\nu})\mid  \nu\unlhd \eta\}\cup \{\varphi_{i_0}(x,b^\alpha_{i_0})\}$$ is consistent (**). 
Then we put $f(\alpha):=i_0$ and $a_\emptyset:=b^\alpha_{i_0}$ (in $C_\alpha$ indexed by $\CT_\alpha$).
Since $a_\emptyset \in D_\alpha$, $C_\alpha$ is s-indiscernible over $D_\alpha$, due to (*) above. The rest of (2) and (4) for $\alpha$ are already satisfied.
Moreover since $C_\alpha$ is s-indiscernible over its root $b^\alpha_{i_0}$, (3) for  $C_\alpha$ also holds by (**). 

It remains to check when $\alpha=\gamma+1$. By the induction hypothesis, $C_\gamma$ is s-indiscernible.
Then by applying Lemma \ref{gettingmseq} to $C_\gamma$ and Morley  $I_{f(\gamma)}$, we find a sequence 
$L=((a_{\eta,i}\mid \eta\in\CT_\alpha)\mid i<\omega)$ with $D_\alpha\equiv D$ such that:  
\be\item[(i)] $C_\gamma =(a_{\eta,0}\mid \eta\in \CT_\alpha)$;

\item[(ii)] $\la a_{\emptyset,i}\mid i<\omega\ra \equiv I_{f(\gamma)}$ and
$\la a_{\emptyset,i}\mid 0< i<\omega\ra $  is Morley over $(\emptyset, C_\gamma D_\alpha)$; and

\item[(iii)]  $L$ is  $D_\alpha$-indiscernible and mutually  s-indiscernible over $D_\alpha$ (so s-indiscernible over $D_\alpha$ as well). 
\ee
In order   to find    a wanted tree $C_\alpha=(a_\eta\mid \eta\in \CT_\alpha)$,  we canonically  put  $a_{\rhd\emptyset}:=L$. 
Then (1) for $\alpha$ is fulfilled due to (ii) above. Now using the same argument in the previous paragraph, we can find $j_0\leq n$ so that 
we put $f(\alpha):=j_0$; $a_\emptyset:=b^\alpha_{j_0}\in D_\alpha$; and  (2),(3),(4) for $\alpha$  again follow due to (iii).
\end{proof}

We are ready to show the main theorem. 

\begin{Theorem}\label{main}
Assume $T$ is NSOP$_1$. Then Kim-forking satisfies existence. Equivalently any complete type
has an $\indk$-Morley sequence in it.
\end{Theorem} 
\begin{proof}
We keep our starting assumption in this section that Kim-forking fails to have existence, and work with chosen  $b_0\dots b_n$ and $I_i$. We will deduce a contradiction under $T$ being NSOP$_1$ (which is preserved regardless of  naming parameters). 

We also work with  Lemma \ref{treeinducedffkfe} and selected $C_\alpha$. There  we take $\kappa$ sufficiently large (to apply Fact \ref{e-r}) and 
in   the tree $C_{<\kappa}=\bigcup_{\alpha<\kappa} C_\alpha=(a_\eta\mid \eta \in \CT_{<\kappa}) $, we consider the sequence
$\la L_\alpha\mid \alpha<\kappa \ra$, where $L_\alpha=(a_{{\zeta_{\alpha+1}}^\smallfrown\la i\ra}\mid i\in\omega). $
By Fact \ref{e-r}, there is an indiscernible  sequence $$J_{<\omega}:=\la J_i=(c_{i,j}\mid j<\omega) \mid i<\omega\ra$$ finitely based on
$\la L_\alpha\mid \alpha<\kappa \ra$.
Notice that due to Lemma \ref{treeinducedffkfe}, we have that all $c_{ij}$ and some $b_i$, say $b_0$  have the same type; and 
$\{\varphi_0(x, c_{i,0})\mid i<\omega\}$ is consistent. Moreover due to s-indiscernibility, each $J_i$ is $J_{>i}$-indiscernible.
Then since  $T$ is NSOP$_1$,
 it follows that  $\{\varphi_0(x, c_{i,1})\mid i<\omega\}$ is consistent as well, by Fact \ref{crcrit}.  Again due to  s-indiscernibility from Lemma \ref{treeinducedffkfe},
we more have that 
 $$J'_{<\omega}:=\la J'_i=(c_{i,j}\mid 1\leq  j<\omega) \mid i<\omega\ra$$  is a Morley array, in particular  strongly  indiscernible. 

Now there is  a model $M$ such that $J'_{<\omega}$ is weak Morley over $M$, due to Lemma \ref{gettingwmly}.  
Hence both $(c_{i,1}\mid i<\omega)$ and $(c_{0,j}\mid 1\leq j<\omega)$ are Morley in $\tp(c_{0,1}/M)$. 
 However since $(c_{0,j}\mid 1\leq j<\omega)\equiv I_0$ as in Lemma \ref{treeinducedffkfe}, 
 $\{\varphi_0(x, c_{0,j})\mid 1\leq j<\omega\}$ is inconsistent. This  contradicts Kim's lemma over a model in NSOP$_1$ $T$.
 We have proved the theorem. 
\end{proof}

We give two remarks before finishing this section. 

\begin{remark}\label{twocommonexist}\begin{enumerate}\item 
If an arbitrary  theory $T$ satisfies  existence, then existence for $\ind^{K}$ follows trivially since $\ind$ implies $\ind^{K}$. On the other hand, if  in $T$  there is  no Morley sequence over $\emptyset$ at all, then existence for $\ind^{K}$ over $\emptyset$ again follows trivially since there is no Kim-dividing formula over $\emptyset$. 

However even if existence fails to hold in an NSOP$_1$ $T$ so that say  $p(x)\in S(\emptyset)$ forks over $\emptyset$, there may exist Morely sequences in  other complete types over $\emptyset$. Thus  even in the case, the proof of Theorem \ref{main} is not trivialized. 
For example, 
as mentioned in the beginning,   Mutchnik in \cite{Mut24} gives a two-sorted  example of NSOP$_1$ theory without  existence.
In one  sort, Morley sequences (over $\emptyset$) do not  exist, while they exist in  the other sort.

\item
Originally, in \cite[Definition 3.13]{KR20} Kaplan and Ramsey define   $\ind^K$   using invariant Morley sequences, which, as proved in  \cite[Theorem 7.7]{KR20}, coincide with  Definition \ref{k-fkdef}(1),(2)  using Morley sequences  when the theory is NSOP$_1$ and the base is a model. 
In any theory,  existence for  $\ind^K$ (as defined in Definition \ref{k-fkdef}(1),(2))    implies  existsnce for $\ind^K$  defined  using  invariant Morley sequences, since every invariant Morley sequence is a Morley sequence. Recently, in an unpublished note, James Hanson shows  $\ind^K$ with respect to invariant Morley sequences satisfies existence in any theory.
\end{enumerate}
\end{remark}

\section{Discussion} \label{discussion}

In this section  we will discuss a couple of issues. First of all it is natural to ask whether  there is a theory $T$ where Kim-forking does not have existence. 
As proved in \S\ref{mainthm}, if exists then it should   not be NSOP$_1$.

As pointed out in \S\ref{intro}, after Fact \ref{basicsoffking}, the typical circle example with the ternary $R$ does not have existence (over $\emptyset$). It is  because 
there are three distinct points $a_0,a_1,a_2$ on the circle such that 
$\models x=x\leftrightarrow R(a_0,x,a_1)\vee  R(a_1,x,a_2) \vee  R(a_2,x,a_0)$, and each $R(a_i,x,a_{i+1})$ (mod $3$) divides over $\emptyset$.  But the dividing of each formula is  witnessed 
{\em not} by a Morley sequence but only by an $\ind^d$-Morley sequence.  Thus this (although  not NSOP$_1$)  does not serve as an example  where Kim-forking fails to have existence.

With respect to this question, at least we observe the following.

\begin{remark}\label{remark:discussion0}\be\item 
Suppose that  a consistent  formula $\psi(x)\in \CL$ forks, witnessed by two consistent formulas $ \varphi(x, a_0)$ and $\phi(x,b_0)$. Namely $\models \psi(x)\leftrightarrow \varphi(x,a_0)\vee\phi(x,b_0)$ and 
dividing of    $\varphi(x,a_0)$ and  $\phi(x,b_0)$ is witnessed by indiscernibles $I=(a_i\mid i<\omega)$, $J=(b_i\mid i<\omega)$, respectively. 
We claim that $I$ (hence $J$ as well) can   neither be $\ind^d$-Morley, nor Morley. (In general, the  reversed sequence of  a non ($\ind^d$-)Morley sequence
can be ($\ind^d$-)Morley. But the proof below says $I$  is not  $\ind^d$-Morley in either  direction. 
Moreover the circle example above  says the  argument below  does not work if three (or more) formulas witness forking of $\psi(x)$):

Note that 
 there are  $ k(\geq 1)$ and $k'\geq 2$ such that 
$\varphi(x, a_0)$ $(k+1)$-divides, 
but not  $k$-divides, via $I$; and $\phi(x,b_0)$ $k'$-divides, via $J$ (*). 
 Consider an  $\CL$-formula 
$$\chi(x_k,\dots, x_0):= \exists z(\psi(z)\wedge \varphi(z,x_k)\wedge \dots \wedge\varphi(z,x_1))\wedge \neg\exists z'  (\psi(z')\wedge \varphi(z',x_k)\wedge \dots \wedge\varphi(z',x_0)).$$
Clearly  $(a_k,\dots, a_0)$ realizes $\chi$. It suffices, by left transitivity of $\ind^d$, to show that $\chi(x_k,\dots ,x_1;a_0)$ divides. 

By compactness there clearly  is an indiscernible $(a^j_0b^j_0\mid j<\omega)$ such that 
$ a^0_0b^0_0= a_0b_0$, and $(b_0^j\mid j\in \omega)\equiv J$. 
Then  $\chi(x_k,\dots ,x_1;a_0)$ $k'$-divides via $(a^j_0\mid j<\omega)$, namely, $\chi(x_k,\dots ,x_1;a^0_0)\wedge \dots\wedge  \chi(x_k,\dots ,x_1;a^{k'-1}_0)$
is inconsistent. Why?  Suppose  not, say realized by $c_k\dots c_1$.
Note that  for each $j<k'$,
$$\psi(z)\wedge \varphi(z,c_k)\wedge \dots \wedge\varphi(z,c_1)\models \neg \varphi (z,a^j_0)\wedge \psi(z)\models \phi(z, b^j_0).$$
Since  $ \psi(z)\wedge \varphi(z,c_k)\wedge \dots \wedge\varphi(z,c_1)$ is consistent, 
$\{\phi(x,b^j_0)\mid j<k'\}$ is then consistent, contradicting (*) above.

\item 
Consider a 2-sorted structure $(\mathbb{N},\mathcal{P}(\mathbb{N}),\in)$ where $\in$ is the usual membership relation. One can find sequences $(B_i)_{i<\omega}$ and $(B'_i)_{i<\omega}$ in $\mathcal{P}(\mathbb{N})$ satisfying $B'_0=\mathbb{N}\setminus B_0$, $B_i\cap B_j=\emptyset$, $ B'_i \cap  B'_j=\emptyset$ for each $i\neq j$, and $|B_i|=|\mathbb{N}\setminus B_i|=|B'_i|=|\mathbb{N}\setminus B'_i|=\aleph_0$ for each $i$. Since any permutation of $\mathbb{N}$ induces an automorphism of the structure,  it easily  follows that $\tp(B_i)=\tp(B'_j)$ for all $i,j$.  Put $B=B_0, B'=B'_0$. Now  $\models x=x\leftrightarrow x\in B\vee x\notin B$.
Moreover
each of the formulas $x\in B$, and $x\notin B$ (equivalent to $x\in B'$) 2-divides over $\emptyset$, witnessed by the sequences $(B_i)_{i<\omega}$ and $(B'_i)_{i<\omega}$, respectively. This is an example of (1), which is a variation of an example failing existence mentioned in \cite{She90}.
\ee
\end{remark}

Now we list the following questions.

\begin{question}\be 
\item Is there a theory where 
Kim-forking does not have  existence? As mentioned in Remark \ref{twocommonexist}(2), if
Kim-forking is defined using invariant Morley sequences, then every theory has Kim-forking existence.
\item In NSOP$_1$ $T$, can one obtain the same  Theorem \ref{main}  even if each $I_i$ in \S3 is weakened to be an  $\ind^d$-Morley sequence? 
\ee
\end{question}

 In addition to the questions above,   
 in all NSOP$_1$ theories 
we are now able to define a new kind of dividing witnessed by   $\indk$-Morley sequences,  due to the main result Theorem \ref{main}.   As  mentioned before \S\ref{intro}, it is intriguing to know how far a theoretical development can be 
made using the new dividing  for NSOP$_1$ theories, although   the new one is known to be  equal to  Kim-dividing and Kim-forking  in NSOP$_1$ theories with existence.

By the way, a quick consequence of  Theorem \ref{main} is  the following.

\begin{remark}
Assume $T$ is NSOP$_1$. Define $A\ind^*_BC$ if there is a model $(B\subseteq) M$ such that $M\indk_BAC$ and $A\indk_MC$. 
\begin{itemize}
\item (Symmetry of $\ind^*$)  Due to symmetry of $\indk$ over  a model, $A\ind^*_BC$ iff $C\ind^*_BA$.
\item (Existence of $\ind^*$) Given $A, B$, by Theorem \ref{main}, there is $(B\subseteq)M$ such that $M\indk_BA$. Then since $B\subseteq M$, we have
$A\indk_MB$. In other words, $A\ind^*_BB$ holds. 
\item  If $T$ has existence then $A\ind^*_BC$ iff $A\indk_BC$ \cite[Proposition 3.5]{CKR20}.
\end{itemize}
However it is unclear whether in $T$,  $\ind^*$ has other properties such as extension, transitivity and so on. 
\end{remark}

\section*{Acknowledgments}
The authors thank Nicholas Ramsey for his explanation about the argument in \cite{CKR20}, which is critically used in the proof of  
 Claim   \ref{gettingmseq}.2. The authors also thank Ikuo Yoneda for his valuable discussion, which helped us refine our article. The authors are grateful to James Hanson and Scott Mutchnik for helpful discussions about their results. The authors thank the anonymous referee for improving the example in Remark \ref{remark:discussion0}(2).

All authors are supported by an NRF of Korea grant 
2021R1A2C1009639. 
The second author is supported by a KIAS individual grant, project no. 6G091801.
The third author is supported by the 2023 and 2024 Yonsei University Post-Doctoral Researcher Supporting Program, project no. 2023-12-0159 and 2024-12-0214, and the G-LAMP Program of the NRF of Korea grant funded by the Ministry of Education (no. RS-2024-00441954).

\end{document}